\documentclass[11pt,a4paper]{article}
\pdfoutput=1
\usepackage[utf8]{inputenc}
\usepackage[T1]{fontenc}
\usepackage[english]{babel}
\usepackage{lmodern,microtype}
\usepackage{amsmath,amssymb,amsthm,mathtools} 
\usepackage{cite} 
\usepackage{eucal,enumitem}

\usepackage[nottoc]{tocbibind} 
\usepackage{tikz,mdframed} 
\usepackage[colorlinks,
  linkcolor=black!40!red,%
  citecolor=green!40!black,%
  urlcolor=black,%
]{hyperref} 

\usepackage[left=2cm,right=3.5cm,top=1.5cm,bottom=1.7cm]{geometry}
\usepackage{changepage} 
\usepackage{marginnote} 
\usepackage{zref-savepos,zref-abspage} 
\usepackage{expl3} 

\colorlet{tn/color/theorem}{green!60!black}
\colorlet{tn/color/claim}{gray!60!white}
\colorlet{tn/color/definition}{black}
\colorlet{tn/color/question}{orange!60!white}
\colorlet{tn/color/conjecture}{orange!60!white}
\colorlet{tn/color/defi}{red!60!black}

\theoremstyle{definition}

\newtheoremstyle{slthm}
{0pt}
{0pt}
{\normalfont}
{}
{\bfseries\footnotesize}
{.\!}
{.6em}
{\thmname{#1}\thmnumber{ #2}\thmnote{\,--#3}}
\theoremstyle{slthm}

\newlength{\mythmbar}
\newlength{\myinsep}
\newlength{\myleftinsep}
\newlength{\mylen}
\newlength{\thmskip}
\mdfsetup{skipabove=\thmskip,skipbelow=\thmskip,usetwoside=false,nobreak}

\newcommand{\tndefmdstyle}[2]{%
  \mdfdefinestyle{tn/#1}{%
    linewidth=\mythmbar,%
    linecolor=tn/color/#2,%
    bottomline=false,%
    topline=false,%
    innerleftmargin=\myleftinsep,%
    leftmargin=-\myinsep,%
    innerrightmargin=\myinsep,%
    innertopmargin=1pt,%
    innerbottommargin=0pt,%
    rightmargin=\myinsep,%
    userdefinedwidth=\mylen}}

\tndefmdstyle{base}{theorem}
\tndefmdstyle{claim}{claim}
\tndefmdstyle{definition}{definition}
\tndefmdstyle{question}{question}

\newmdtheoremenv[style=tn/base]{theorem}{Theorem}

\newcommand{\tnmdenv}[3]{%
  \newmdtheoremenv[style=#1]{#2}[theorem]{#3}}

\tnmdenv{tn/base}{lemma}{Lemma}
\tnmdenv{tn/base}{proposition}{Proposition}
\tnmdenv{tn/claim}{claim}{Claim}
\tnmdenv{tn/claim}{example}{Example}
\tnmdenv{tn/base}{corollary}{Corollary}
\tnmdenv{tn/base}{fact}{Fact}
\tnmdenv{tn/question}{question}{Question}
\tnmdenv{tn/question}{problem}{Problem}
\tnmdenv{tn/question}{conjecture}{Conjecture}
\tnmdenv{tn/definition}{definition}{Definition}
\tnmdenv{tn/definition}{remark}{Remark}
\tnmdenv{tn/definition}{hole}{Hole}

\tnmdenv{tn/base}{teorema}{Teorema}
\tnmdenv{tn/base}{lema}{Lema}
\tnmdenv{tn/base}{proposicao}{Proposição}
\tnmdenv{tn/base}{corolario}{Corolário}
\tnmdenv{tn/base}{fato}{Fato}
\tnmdenv{tn/claim}{afirmacao}{Afirmação}
\tnmdenv{tn/claim}{exemplo}{Exemplo}
\tnmdenv{tn/question}{questao}{Questão}
\tnmdenv{tn/question}{problema}{Problema}
\tnmdenv{tn/question}{conjectura}{Conjectura}
\tnmdenv{tn/definition}{definicao}{Definição}
\tnmdenv{tn/definition}{nota}{Nota}
\tnmdenv{tn/definition}{buraco}{Buraco}

\newtheoremstyle{tnUnnumb}
{0pt}
{0pt}
{\normalfont}
{}
{\bfseries\footnotesize}
{.}
{.5em}
{\thmnote{#3}}

\theoremstyle{tnUnnumb}
\newmdtheoremenv[style=tn/base]{unnumtheorem}{Theorem}
\newmdtheoremenv[style=tn/base]{unnumconjecture}{Conjecture}

\newcommand{\newproofenv}[3]{%
  \newenvironment{#1}[1][#2]{%
    \renewcommand{\qedsymbol}{\hfill{#3}}\begin{proof}[#2]%
  }{%
    \end{proof}\renewcommand{\qedsymbol}{\oldqed}%
  }%
}

\newproofenv{claimproof}{Proof}{$\diamond$}
\newproofenv{proofsketch}{Proof sketch}{$\bigtriangleup$}

\newproofenv{afproof}{Prova}{$\diamond$}
\newproofenv{esboco}{Esboço da prova}{$\bigtriangleup$}

\newcommand{\defistyle}[1]{\textcolor{definition}{\textsl{#1}}}
\newcommand{\defi}[1]{%
  \sidepar{{\tiny#1}}%
  {\defistyle{#1}}}

\ExplSyntaxOn\makeatletter
\int_new:N \g_tassio_note_int
\seq_new:N\g_tassio_note_seq
\renewcommand{\defi}{\@dblarg\tn@defi} 
\renewcommand{\defistyle}[1]{%
  \textcolor{tn/color/defi}{\textsl{#1}}%
}
\def\tn@defi[#1]#2{%
  \int_gincr:N \g_tassio_note_int
  \bool_if:nTF
  {
   \int_compare_p:n
    {
     \zposy{tassionotepos\int_eval:n{\g_tassio_note_int}}
     =
     \zposy{tassionotepos\int_eval:n{\g_tassio_note_int +1}}
    }
   &&
   \int_compare_p:n
   {
    \zref@extractdefault { tassionotepage\int_eval:n{\g_tassio_note_int } }{abspage}{-1}
    =
    \zref@extractdefault { tassionotepage\int_eval:n{\g_tassio_note_int +1} }{abspage}{-1}
   }
  }
   {
    \seq_gput_right:Nn \g_tassio_note_seq {#1}
   }
   {
    \seq_gput_right:Nn \g_tassio_note_seq {#1}
    \marginnote{\fontsize{7pt}{5pt}\selectfont%
      \baselineskip=.6\baselineskip 
      \lineskip=-1.2pt 
      \seq_use:Nn \g_tassio_note_seq {,\ }}
    \seq_gclear:N \g_tassio_note_seq
   }
  \zref@label {tassionotepage\int_use:N\g_tassio_note_int}\zsaveposy {tassionotepos\int_use:N\g_tassio_note_int}
  \defistyle{#2}}

\ExplSyntaxOff

\makeatother

\makeatletter
\renewcommand{\@biblabel}[1]{\textcolor{gray}{[\,}#1\textcolor{gray}{\,]}}
  \renewcommand\@openbib@code{
    \setlength\labelwidth{2cm}%
    \setlength{\itemindent}{0cm}%
    \setlength{\leftmargin}{0cm}%
    \setlength\labelsep{.8em}%
    }
\makeatother

\newcommand{\eps}{\varepsilon}

\newcommand{\deq}{\coloneqq}
\newcommand{\littleo}{\mathrm{o}}

\newcommand{\calb}{{\mathcal{B}}}
\newcommand{\calc}{{\mathcal{C}}}

\newcommand{\calj}{{\mathcal{J}}}

\newcommand{\bbe}{{\mathbb{E}}}

\newcommand{\bbp}{{\mathbb{P}}}

\definecolor{mygreen}{rgb}{0.01, 0.75, 0.24}

\newcommand{\afigure}[3]{%
  \begin{figure}
    \begin{center}
      \includegraphics{#1}
      \caption{#2}\label{#3}
    \end{center}
  \end{figure}
}

\newcommand{\ttk}{\ensuremath{\mathrm{TT}_k}}

\newcommand{\ttt}{\ensuremath{\mathrm{TT}_3}}
\newcommand{\mm}{\ensuremath{m_2}}
\newcommand{\arr}{\to}
\newcommand{\narr}{\not\to}

\DeclareMathOperator{\type}{type}
\newcommand{\length}{e}

\let\emptyset\varnothing
\newcommand{\Carr}{{\ensuremath{\vec C}}}

\newcommand{\Farr}{{\ensuremath{\vec F}}}
\newcommand{\Garr}{{\ensuremath{\vec G}}}
\newcommand{\Harr}{{\ensuremath{\vec H}}}

\newcommand{\Tarr}{{\ensuremath{\vec T}}}

\def\({\left(}
\def\){\right)}
\let\:=\colon

\let\geq\geqslant
\let\leq\leqslant
\let\ge\geqslant
\let\le\leqslant

\newcommand{\oramsey}{\ensuremath{\vec R}\,}

\newcommand{\Aplain}[1]{\ensuremath{(A_{#1})}}
\newcommand{\Bplain}[1]{\ensuremath{(B_{#1})}}
\newcommand{\Cplain}[1]{\ensuremath{(C_{#1})}}
\newcommand{\A}[1]{\hyperref[config_A]{\Aplain{#1}}}
\newcommand{\B}[1]{\hyperref[config_B]{\Bplain{#1}}}
\newcommand{\C}[1]{\hyperref[config_B]{\Cplain{#1}}}

\newcommand{\Aellminus}{\A{\ell-1}}

\newcommand{\Nenadov}{Nenadov} 
\newcommand{\Person}{Person} 
\newcommand{\Skoric}{\v Skori\'c} 
\newcommand{\Steger}{Steger} 

\newcommand{\marginarrow}{\makebox[0em][r]{\raisebox{.1ex}{$\dashrightarrow{}$}}}

\newcommand{\case}[1]{\par\smallskip\noindent\marginarrow\textbf{Case #1. }\ignorespaces}

\newcommand{\subcase}[1]{\par\smallskip\noindent\marginarrow\textit{Case (#1). }\ignorespaces}

\newcommand{\thr}[1]{p_{#1}}
\newcommand{\sseq}{\subseteq}
\newcommand{\abs}[1]{|#1|}

\begin{document} 
\begin{center} 
  {\Large\noindent\textbf{Orientation Ramsey thresholds for cycles and
      cliques}\footnote{This study was financed in part by the
      Coordena\c c\~ao de Aperfei\c coamento de Pessoal de N\'ivel
      Superior, Brasil (CAPES), Finance Code 001.}
  }

  \smallskip%

  \noindent
  Gabriel Ferreira Barros\quad
  Bruno Pasqualotto Cavalar\footnote{FAPESP 2018/05557-7. Part of this work
    was completed while he was a master's student at IME-USP.}

  Yoshiharu Kohayakawa\footnote{CNPq
    (311412/2018-1, 423833/2018-9) and FAPESP (2018/04876-1).}\quad
  T\'assio Naia\footnote{FAPESP 2019/04375-5 and~2018/04876-1.  FAPESP
    is the S\~ao Paulo Research Foundation.  CNPq is the National
    Council for Scientific and Technological Development of Brazil.}

\end{center} 

\medskip

\begin{adjustwidth}{3\parindent}{3\parindent} 
  \noindent\small\textbf{Abstract.}
  If $G$ is a graph and $\vec H$~is an oriented graph, we write
  $G\to \vec H$ to say that every orientation of the edges of~$G$
  contains~$\vec H$ as a subdigraph. We consider the case in
  which~$G=G(n,p)$, the binomial random graph. We determine the
  threshold~$p_{\vec H}=p_{\vec H}(n)$ for the property
  $G(n,p)\to \vec H$ for the cases in which~$\vec H$ is an acyclic
  orientation of a complete graph or of a cycle.
\end{adjustwidth} 

\bigskip

\noindent\textbf{A Ramsey-type property}.
For each (undirected) graph $G$ and oriented graph \Harr, we write
$G\arr \Harr$ to mean that every orientation of~$G$ contains a copy
of~\Harr; the \defi{orientation Ramsey number}~$\oramsey(\Harr)$ is
$\inf \{\, n : K_n\arr \Harr \,\}$.
This parameter has been investigated in a number of articles~\cite[among
others]{Redei34,Stearns59:voting,ErdosMo64,Chung82:trees,wormald83,LinSaksSos83,Thomason1986,HagThom91,Lu91:claws,Lu93:Claws,Lu98:claws,havet2002:trees,Sahili2004,KMO10:sumner_exact,KMO11:sumner_approximate,MyNa2018:unavoidable,DrossHavet18},
most of which concern a~conjecture of~Sumner~\cite{wormald83}.
\emph{Sumner's universal tournament conjecture}  states that~$\oramsey(\Tarr)\le 2e(\Tarr)$
for every oriented
tree~$\Tarr$; this has been confirmed for all sufficiently large trees
by K\"uhn, Mycroft and
Osthus~\cite{KMO11:sumner_approximate,KMO10:sumner_exact}; see
also~\cite{BangJGutin01,Naia18:thesis}.

\medskip\noindent\textbf{Thresholds}. Thresholds for Ramsey-type
properties are widely studied as well (see,
e.g.,~\cite{janson00:_random_graph,NPSS17:ramsey_fram} and the many
references therein).
We call
$p_{\Harr}=p_{\Harr}(n)$ a \defi{threshold} for $G(n,p)\arr\Harr$ if
\begin{equation*}
  \bbp\bigl[\,G(n,p)\arr \Harr\,\bigr] =
\begin{cases}
  0 & \text{if $p\ll p_{\Harr}$}\\
  1 & \text{if $p\gg p_{\Harr}$},
  \end{cases}
\end{equation*}
where $a\ll b$ (or, equivalently, $b\gg a$) means
$\lim_{n\to\infty\,} a_n/b_n\to 0$.
As is customary, we speak of `the threshold~$p_{\Harr}$', since
$p_{\Harr}$ is unique within constant factors.
If~$\Harr$ is acyclic, then the property
$G(n,p)\arr \Harr$ is non-trivial and monotone, and
hence~\cite{bollobas87:_thres_funct} it has a
threshold~$p_{\Harr}=p_{\Harr}(n)$. The regularity method can be
used to give an upper bound for~$p_{\Harr}=p_{\Harr}(n)$ (it
suffices to combine ideas from \cite[Section 8.5]{janson00:_random_graph} and, say, \cite{conlon14:_klr}). For an alternative
approach giving the same upper bound, based on the methods of
\cite{nenadov_steger_2016}, see~\cite{cavalar2018ramsey}.
For any graph or digraph $G$,
the \defi[max.\ density]{maximum density} and
(when~$v(G)\ge 3$) the \defi[max.\ $2$-density]{maximum $2$-density}
of~$G$ are, respectively,
\begin{equation*}
  m(G)   \deq
  \max_{\substack{J\subseteq G\\ v(J) \ge 1}}\,\frac{e(J)}{v(J)}
  \qquad\text{and}\qquad
  \mm(G) \deq
  \max_{\substack{J\subseteq G\\ v(J) \ge 3}}\,
  \frac{e(J)-1}{v(J)-2}.
\end{equation*}

\begin{theorem}
  \label{thm:cav18_threshold}
  Let $\Harr$ be an acyclically oriented graph.
  There exists a constant $C = C(\Harr)$ such that,
  if
  $p \ge Cn^{-1/\mm(\Harr)}$,
  then~$\bbp\bigl[ G(n,p) \arr \Harr \bigr]\to1$ as $n\to\infty$.
\end{theorem}

\noindent\textbf{Contribution}.
We determine the orientation Ramsey threshold
for all acyclic orientations of the complete graph~$K_t$
and cycle~$C_t$, for each~$t\geq 3$.
We also determine the threshold for certain oriented bipartite graphs.
We call a digraph~\Harr\ \defi{anti-directed}
if each vertex in~\Harr\ has either no inneighbours or no~outneighbours
(so~\Garr\ is bipartite and all arcs point to the same part).

\begin{theorem}
    \label{t:thresholds}
    If $\Harr$ is an acyclic orientation of~$K_t$ or~$C_t$, then
    \begin{align*}
      p_{\Harr} (n)  & =
      \begin{cases}
        \displaystyle n^{-1/m(K_4)}      &  \text{if $t=3$}  \\
        \displaystyle n^{-1/\mm(\Harr)}   & \text{if $t\ge 4$}
      \end{cases} \\
      \intertext{is the threshold for $G(n,p)\arr \Harr$. Moreover, if
      \Harr\ is an anti-directed orientation of a strictly
      $2$-balanced graph~$H$ such that $\delta(H)\ge 2$
      and~$\mm(H) - \lfloor\mm(H)\rfloor \leq 1/2$, then}
      p_{\Harr}
      (n) & = n^{-1/\mm(\Harr)}
    \end{align*}
is the threshold for $G(n,p)\arr \Harr$.
\end{theorem}

In view of Theorem~\ref{thm:cav18_threshold}, to prove
Theorem~\ref{t:thresholds} (except for the case in which~$\vec H$ is
an orientation of~$K_3$), it suffices to prove the so called
$0$-statement, that is, it is enough to show that
if~$p\ll n^{-1/m_2(\vec H)}$, then $G(n,p)\arr{\vec H}$ holds with
vanishing probability.  Our proof of this $0$-statement uses recent
advances in the study of Ramsey-type thresholds: a framework developed
by \Nenadov, \Person, \Skoric\ and~\Steger~\cite{NPSS17:ramsey_fram}
(outlined below) and structural results of~Barros, Cavalar, Mota
and~Parczyk~\cite{BPCOMP20:antiramsey}.

We need only a simplified version of the results in~\cite{NPSS17:ramsey_fram}
(see Definitions 10 and~11 in\cite{NPSS17:ramsey_fram}).
Let $G$ and~$H$ be graphs, where~$\delta(H)>1$.
An edge $e\in E(G)$ is \defi[$H$-closed edge]{$H$-closed}
if $e$~belongs to at least two copies of~$H$ in~$G$.
A copy of~$H$ in~$G$ is \defi[$H$-closed copy]{$H$-closed}
if at least three of its edges are $H$-closed, and
$G$ is \defi[$H$-closed graph]{$H$-closed} if
all vertices and edges of~$G$ lie in copies of~$H$
and every copy of $H$ in~$G$ is $H$-closed.
Finally, $G$ is an \defi{$H$-block} if $G$ is $H$-closed and
for each proper non-empty subset $E'\subsetneq E(G)$
there exists a copy $H'$ of~$H$ in~$G$
such that $E(H')\cap E' \neq \emptyset$
and~$E(H')\setminus E'\neq \emptyset$.

\medskip
\begin{theorem}[{\cite[Corollary 13]{NPSS17:ramsey_fram}}]
  \label{t:npss17_block}
    \label{cor:npss17_block}
    Let $H$ be a strictly $2$-balanced graph with at least $3$ edges
    such that $H$ is not a matching.
    If $p \ll n^{-1/\mm(H)}$, then
    with high probability every
    $H$-block~$F$ of $G(n,p)$ satisfies~$m(F) < \mm(H)$.
\end{theorem}

Since complete graphs and cycles are strictly $2$-balanced,
Theorem~\ref{cor:npss17_block} reduces the proof of the $0$-statement
of the case $t\ge 4$ in Theorem~\ref{t:thresholds}
to showing that $G\narr \Harr$ for every graph~$G$
whose $H$-blocks have maximum density strictly below~$\mm(H)$.
This is achieved for cycles using results from~\cite{BPCOMP20:antiramsey},
whereas for tournaments and anti-directed graphs, as well as
for the case $t=3$ of Theorem~\ref{t:thresholds} we
use ad~hoc methods (see Theorems~\ref{thm:bound_ttt},
\ref{t:tourn-dens-lbound}~and~\ref{t:anti-directed}).
Theorem~\ref{t:thresholds} is proved in~Section~\ref{s:main-proofs}.

\bigskip\noindent \textbf{Remark.} Other Ramsey-type
properties for directed graphs include
requiring copies to be induced~\cite{cochand1989few,kohayakawa1996ramsey,brightwell1993ramsey} and
allowing colourings plus orientations~\cite{BucHebLetzSud2020,Buc2019}.

\section{Auxiliary definitions and results}
\label{s:tools}
We follow standard notation
(see, e.g.,~\cite{janson00:_random_graph,diestel00:_graph}).
A $k$-path is a path with $k$~vertices; $k$-cycles are defined similarly.
A~\defi[directed path or~cycle]{directed} $k$-path is an oriented path~$v_1\arr \cdots\arr v_k$.
A~\defistyle{directed} $k$-cycle is oriented as $v_1\arr\cdots\arr v_k\arr v_1$.
Let~\Garr\ be an oriented graph.
A maximal directed path in~\Garr\ is called a~\defi{block}.
A path or block is \defi{long} if it has at least~$3$~edges.
The following exercise is left to the reader.

\begin{lemma}\label{l:d-degenerate}
  If $G$ is a graph, then $\delta(J) \le 2m(G)$
  for each $J \subseteq G$ (i.e., $G$ is $2m(G)$-degenerate).
\end{lemma}

Let $G$ and~$H$ be graphs, and let~$\Harr$ be an orientation of~$H$.
We denote by~$\calc_{H}(G)$ the \defi[]{edge intersection graph} of~$H$ in~$G$, whose
vertices correspond to copies of~$H$ in~$G$ and
 whose edges join distinct copies which share a common~edge in~$G$.
An~\defi{$H$-component} is a subgraph of~$G$
formed by the union of all copies of~$H$
in some connected component of $\calc_H(G)$.
Note that $G\narr \Harr$
if and only if each $H$-component of $G$
admits an $\Harr$-free orientation.
Let~$G$ and~$H$ be graphs and
let $C$ be an $H$-component of~$G$.
If $H_1$ is an arbitrary copy of~$H$ in~$C$, then there exists a
sequence $H_1 \subseteq H_2\subseteq \cdots\subseteq H_t = C$ with the
following property. For each $i\in[t-1]$, there exists a copy~$H'$ of~$H$
such that $H' \not \subseteq H_i$,  $E(H') \cap
E(H_i) \neq \emptyset$ and~$H_{i+1} = H_i \cup H'$.
We say that $(H_1,\ldots,H_t)$ \defi{constructs}~$C$,
and call~$(H_1,\ldots,H_t)$ a~\defi[construction seq.]{construction sequence} of~$C$.
For each~$i\in[t-1]$, we say that a vertex or edge of~$H_{i+1}$ is
\defi[new vertex or edge]{new} in~$H_{i+1}$ if it is not contained in~$H_i$, and say that
$F\subseteq H_{i+1}$ is \defi[new graph]{new} (in~$H_{i+1}$) if~$F$ contains a new edge in~$H_{i+1}$.
 Moreover, if $H_1'$ is a copy of~$H$ in~$H_i$, then there exists a
construction sequence~$(H_1',\ldots, H_j')$ of~$H_i$ starting
with~$H_1'$, and hence a construction
sequence~$(H_1',\ldots,H_j',H_{i+1},\ldots,H_t)$ of~$C$.

Let $G$ be a graph and suppose $E\subseteq E(G)$. We write \defi{$G[E]$} for the
subgraph of~$G$ consisting of the edges in~$E$ and the vertices in~$G$
which are incident with those edges.
We call $H$ \defi{strictly $2$-balanced} if
$\mm(F) < \mm(H)$ for each proper subgraph
$F \sseq H$.

\begin{lemma}[{\cite[Lemma 14]{NPSS17:ramsey_fram}}]
  \label{l:decomp-in-blocks}
  Let $G$ and~$H$ be graphs. If~$G$ is $H$-closed,
  then~$E(G)$ admits a partition $\{E_1,\ldots,E_k\}$
  such that
  $G[E_1], \ldots, G[E_k]$ are $H$-blocks
  and each copy of~$H$ in~$G$ lies entirely
  in one of these $H$-blocks.
\end{lemma}

Let \Harr\ be an orientation of a graph~$H$.
We say~\Harr\ is \defi{$2$-Ramsey-avoidable}
if for all $e,f\in E(H)$, every orientation of~$e,f$
can be extended to an \Harr-free orientation of~$H$.

\begin{remark}\label{r:2-ramsey-avoidable}
  Let $k\ge 4$.
  If \Harr\ is either an orientation of~$C_k$, a transitive tournament~\ttk,
  or an anti-directed orientation of a graph~$H$ with~$\delta(H)>1$,
  then \Harr\ is $2$-Ramsey-avoidable.
\end{remark}

\begin{proof}
  Let $H$ be the underlying graph of~\Harr.
  In each of the following cases,
  let $e,f\in E(H)$ be chosen and oriented arbitrarily;
  it suffices to complete an \Harr-free orientation of~$H$.

  Suppose \Harr\ is an orientation of~$C_k$.
  Note that we can complete the orientation of~$e,f$ to orientations $\Carr_1,\Carr_2$ of~$C_k$ such that
  $\Carr_1$ has a block of length at least $k-1\geq 3$ and
  $\Carr_2$ has no long block.
  If \Harr\ has a block of length at least $k-1$,
  then we pick~$\Carr_2$, else we pick~$\Carr_1$.

  If $\Harr\simeq\ttk$,
  we complete the orientation of~$K_k$
  so that it contains a directed triangle
  (some triangle in~$H$ has
  at most one edge already oriented).

  In the remaining case (anti-directed graph),
  we complete the orientation of~$H$ forming a directed $3$-path
  (since $k\ge 4$,
  some~$v\in V(H)$ is incident with precisely one of~$e,f$,
  while $\delta(H)>1$ implies
  some other edge incident with $v$ has not been oriented).
\end{proof}

Remark~\ref{r:2-ramsey-avoidable} will be used with the next lemma and
Theorem~\ref{t:npss17_block} to establish our main results.

\begin{lemma}\label{l:component-to-blocks}
  Let $G$ be a graph and let \Harr\ be \defi{$2$-Ramsey-avoidable}.
  If $B\narr \Harr$ for each $H$-block $B$ of $G$, then $G\narr\Harr$.
\end{lemma}

\begin{proof}
  Let $H$ be the underlying graph of~\Harr.
  To show that $G$ admits an \Harr-free orientation,
  we may assume each edge of~$G$ lies in a copy of~$H$
  (the orientation of other edges is irrelevant).

  Let $G_0=G$ and, for each $i = 1,2,\ldots$ proceed as follows.
  If $G_{i-1}$ is $H$-closed, then stop, set~$m\deq i -1$ and $F\deq G_{m}$.
  Otherwise, some copy $F_i$ of $H$ in~$G_{i-1}$
  has at most two $H$-closed edges in~$G_{i-1}$.
  Form $G_i$ by deleting from $G_{i-1}$ each non-$H$-closed edge of~$F_i$,
  and then each isolated vertex.
  Note that $G_{i-1} = G_i\cup F_i$ and
  that each $e\in E(G_i)$ lies in some copy of~$H$.

  Note that $F$ is $H$-closed.
  By Lemma~\ref{l:decomp-in-blocks},
  $F$ can be partitioned into a collection~$\calb$ of edge-disjoint $H$-blocks
  such that each copy of~$H$ in~$F$ lies entirely in some~$B\in\calb$.
  By assumption, $B\narr\Harr$ for each $B\in\calb$,
  so $F$ admits a \Harr-free orientation~\Farr\
  (the disjoint union of \Harr-free orientations of each $B\in\calb$).

  Finally, we extend $\Garr_m\deq \Farr$
  to an \Harr-free orientation~$\Garr_0$ of~$G$.
  For each $i= m,m-1,\ldots,1$, let $\Garr_{i-1}$ extend~$\Garr_i$
  by orienting the edges~$E(F_i)\setminus E(G_i)$
  so that~$F_i$ is \Harr-free
  (this is possible because~\Harr\ is $2$-Ramsey-avoidable).
  Clearly, no copy of~$H$ in~$G$ induces \Harr\ in~$\Garr$,
  so~$G\narr \Harr$.
\end{proof}

\section{Transitive triangles}
\label{s:ttt}

Let $\ttt$ be the transitive triangle. In this section
we show that the upper bound for~$\thr{\ttt}(n)$ given
in~Theorem~\ref{thm:cav18_threshold} is not tight.

\begin{theorem}
    \label{thm:bound_ttt}
    The function $\thr{\ttt}(n) = n^{-1/m(K_4)}$
    is the threshold for $G(n,p)\arr\ttt$.
\end{theorem}

Let $W_5$ be the graph we obtain by adding to~$C_4$ a new universal vertex.

\begin{proposition}\label{p:K3-paths}
  If $G$ is a $K_3$-component such that
   $uw,vw\in E(G)$ and~$uv\notin E(G)$, then
   there exists $J\subseteq G$ such that
   either $v(J)=6$ and~$e(J)=9$ or
   $J+ uv$ is isomorphic to $K_4$~or~$W_5$.
\end{proposition}

\begin{proof}
  Let $F_1\cdots F_s$ be a shortest path in~$\calc_{K_3}(G)$
  such that $uw\in E(F_1)$ and $vw\in E(F_s)$.
  It suffices to show the following.
  \begin{itemize}
  \item If $s=2$, then $J\deq F_1\cup \cdots \cup F_s$
    satisfies $J + uv  \simeq  K_4$.
  \item If $s=3$, then $J\deq F_1\cup \cdots \cup F_s$
    satisfies $J + uv\simeq W_5$.
  \item If $s\ge 4$, then $J\deq F_1\cup F_2\cup F_3\cup F_4$ satisfies $v(J)=6$ and~$e(J)=9$.
  \end{itemize}

It is simple to check that the following hold by the choice of~$F_1\cdots F_s$.
\begin{enumerate}[label=(\roman*)]
\item $|E(F_i)\cap E(F_{i-1})|=1$ for all~$i\in[s]\setminus\{1\}$;\label{it:K3-paths/ii}
\item $\bigl|E(F_i)\cap \bigcup_{j < i} E(F_j)\bigr| = 1$
  for each~$i\in[s]\setminus\{1\}$; and\label{it:K3-paths/iii}
\item Each $e\in E(G)$ belongs to at most two triangles in~$F_1\cup \cdots \cup F_s$.\label{it:K3-paths/i}
\end{enumerate}
The statement for $s=2$ follows by~\ref{it:K3-paths/ii} since~$F_1\simeq K_3$.
If~$s=3$, then $v(J)=5$ (by~\ref{it:K3-paths/ii}~and~\ref{it:K3-paths/iii}), so $J + uv\simeq W_5$.
 By~\ref{it:K3-paths/ii}, for each $i\in[s]\setminus \{1\}$ we have
$|V(F_i) \setminus \bigcup_{j\in[i-1]} V(F_j)| \le 1$,
so $v(F_1\cup \cdots \cup F_s) \le s + 2$.
Moreover,  \ref{it:K3-paths/iii} implies~$e(F_i) =2i + 1$ for each~$i\in[s]$.
If~$s=4$, then $e(J) = 9$. Clearly $5 \le v(J)\le 6$; note that~$v(J) \neq 5$ as
otherwise there exists~$e\in E(J)$ which belongs to three distinct triangles in~$J$,
contradicting~\ref{it:K3-paths/i}.
\end{proof}

Let~$(H_1,\dots,H_t)$ be a~$K_3$-component.
For each~$i \in [t-1]$, either
\Aplain{} there are two new edges in~$H_{i+1}$ and one new vertex in~$H_{i+1}$;
\Bplain{} there are two new edges in~$H_{i+1}$ and~$V(H_{i+1})= V(H_{i})$; or
\Cplain{} there is exactly one new edge in~$H_{i+1}$ and~$V(H_{i+1}) = V(H_i)$.
A graph~$H$ is~\defi{$AB$-constructible} if
no construction sequence of a $K_3$-component of~$H$
contains a step of type~\Cplain{}.

\begin{proposition}\label{p:AB-constructible}
  If a graph~$H$ is $AB$-constructible, then $H\narr \ttt$.
\end{proposition}

\begin{proof}
  We may assume that $H$ is itself a single $K_3$-component $(H_1,\ldots,H_t)$,
  as edges which do not belong to a copy of~$K_3$ in~$H$
  can be arbitrarily oriented
  and distinct $K_3$-components may be independently oriented.
  First note that, at each step, exactly one new copy of~$K_3$ is added.
  This is clearly true for steps of type \Aplain{}.
  Moreover, it is easy to see that
  if $H_{\alpha+1}$ is created by a step of type \Bplain{}
  and the new edges create two distinct copies of~$K_3$ in~$H_{\alpha +1}$,
  then $H$ admits a construction sequence with a step of type~\Cplain{},
  a contradiction.
  We orient $H_1$ forming a directed triangle
  and, for each $\alpha\in[t-1]$,
  orient the two new edges in~$H_{\alpha+1}$
  so as to form a new directed triangle.
  The resulting orientation is \ttt-free.
\end{proof}

Our final ingredient is the following classical result
(see, e.g.,~\cite{janson00:_random_graph}).

\begin{theorem}[{\cite{janson00:_random_graph}}]
  \label{thm:threshold_subgraph}
    Let $H$ be a fixed graph.
    Then
    \begin{equation*}
        \lim_{n \to \infty}
        \bbp[H \sseq G(n,p)]
        =
        \begin{cases}
            1, \text{ if } p \gg n^{-1/m(H)},\\
            0, \text{ if } p \ll n^{-1/m(H)}.
        \end{cases}
    \end{equation*}
\end{theorem}

\newcommand{\whp}{with high probability}%
\begin{proof}[Proof of Theorem~\ref{thm:bound_ttt}]
    If~$p \gg n^{-2/3}$, then
    $K_4\subseteq G(n,p)$ with high probability
    by~Theorem~\ref{thm:threshold_subgraph};
    hence $G(n,p) \arr \ttt$ \whp\
    (as~$K_4\arr \ttt$).

  Now suppose that~$p \ll n^{-2/3}$.
  Let $\mathcal{E}$ be the event that $G(n,p)$ is not $AB$-constructible.
  By Proposition~\ref{p:AB-constructible}, it suffices
  to show that~$\bbp[\mathcal{E}]=\littleo(1)$.
  Let $\calj$ be set of all nonisomorphic graphs of order~$6$ and size~$9$.
  By Proposition~\ref{p:K3-paths}, every $K_3$-component of~$G(n,p)$ which
  is not $AB$-constructible contains either $K_4$, $W_5$ or some~$J\in\calj$.
  Using Markov's inequality, we have
  \[
    \begin{aligned}\bbp[\mathcal{E}] &\le \bbp[K_4 \subseteq G(n,p)]
      +\bbp[W_5 \subseteq G(n,p)]
      +\sum_{J\in\calj} \bbp[J \subseteq G(n,p)] \\
      &\le \sum_{J \in \{K_4,W_5\}\cup\calj}\bbe\bigl[\, |\{J' \subseteq G(n,p) : J' \simeq J\}|\,\bigr]
      \,\,\,\le\,\,\,   n^4p^6 + n^5p^8 + |\calj|n^6p^9.
    \end{aligned}
  \]
  Since~$p \ll n^{-2/3}$ and~$|\calj|=\Theta(1)$, we have~$\bbp[\mathcal{E}]=\littleo(1)$.
\end{proof}

\section{Graphs with low maximum 2-density}

The following sections show that $G\narr\Harr$ for some classes of oriented graphs,
when \Harr\ has at least four vertices and $m(G)<\mm(\Harr)$.

\subsection{Transitive Tournaments}
\label{s:trans-tourn}

We denote a tournament on~$k$ vertices by~\defi{$T_k$},
writing \defi{\ttk}\ if it is transitive.

\begin{theorem}\label{t:tourn-dens-lbound}
  If $k\ge 4$ and $G$ is a graph with $m(G) < \mm(K_{k})$,
  then~$G\narr \ttk$.
\end{theorem}

\begin{proof}
  The proof is
  by induction on $n:=v(G)$.
  The case~$n=1$ is trivial.
  Assume~$n \ge 2$ and that
  $G'\narr\ttk$ whenever~$m(G')<\mm(K_k)$
  and~$v(G')<n$.
  By Lemma~\ref{l:d-degenerate}, $\deg_G(u)\le k$
  for some~$u\in V(G)$. Let~$G'=G-u$,
  so~$m(G') < \mm(K_k)$ and
  $G'$~admits a $\ttt$-free orientation~$\Garr$.
  We shall extend $\Garr$ to an orientation of~$G$ such that
  each~$T_k$ containing~$u$ has a directed cycle.

  We may assume that~$u$ lies in some copy of~$K_k$, say~$K$;
  so~$\deg(u) \ge k-1$.
  If~$K$ is the only copy of~$K_k$ containing~$u$, then
  choose two vertices~$v,w \in V(K-u)$
  and orient the edges~$uv$ and~$uw$ so that~$\{u,v,w\}$
  induces a directed triangle.
  Otherwise let~$K'$ be some~$K_k$ containing~$u$ other than~$K$.
  Hence we must have~$\deg(u)=k$.
  Let~$v$ be the unique vertex in~$V(K) \setminus V(K')$,
  and~$w$ be the unique vertex in~$V(K') \setminus V(K)$.
  Since~$k-2\ge 2$, there are at least two vertices~$x$ and~$y$
  in~$V(K \cap K') \setminus \{u\}$.
  Orient the edges $uv$, $ux$, $uw$ and $uy$
  so that each of~$\{u,v,x\}$ and~$\{u,w,y\}$ induces a directed triangle.
  Since every $K_k$ containing~$u$ has at least three vertices in~$\{v,w,x,y\}$,
  the partial orientation of each~$K_k$ contains a directed cycle.
  Any remaining un-oriented edge may be arbitrarily oriented.
\end{proof}

\subsection{Anti-directed digraphs}
\label{s:anti-directed}

We now turn to anti-directed orientations of~$K_{t,t}$, $C_{2t}$ and other bipartite graphs.

\begin{theorem}\label{t:anti-directed}
  Let $G$ and~$H$ be graphs, where~$\delta(H)\ge 2$. If $\Harr$ is an
  anti-directed orientation of~$H$ and~$m(G) < \delta(H) -1/2$, then
  $G\narr \Harr$.
\end{theorem}

\begin{proof}
  We proceed by induction on~$v(G)$.
  If~$v(G) \le 2$, then~$G\narr\Harr$.
  Let $\delta\deq\delta(H)$.
  By~Lemma~\ref{l:d-degenerate}, there exists~$v\in V(G)$ with
  $\deg(v) = \delta(G)\le 2\delta -2$.
  By~induction, $G-v\narr\Harr$.
  Fix an $\Harr$-free orientation of~$G-v$, orient
  $\lfloor \delta(G)/2\rfloor$~edges incident with~$v$ towards~$v$ and the
  remaining $\lceil \delta(G)/2 \rceil$~edges away from~$v$.
  Note that
  any copy of~$H$ in~$G$ containing~$v$
  necessarily has two edges incident with~$v$ oriented in opposite directions,
  since~$\delta(H)\ge 2$
  and~$\lceil\delta(G)/2\rceil \le \delta(H)-1$.
\end{proof}

\begin{corollary}\label{cor:anti-directed}
  Let $G$ and~$H$ be graphs such that
  $H$ is strictly $2$-balanced,
  $\delta(H)\ge 2$
  and~$\mm(H) - \lfloor \mm(H) \rfloor \le 1/2$.
  If $\Harr$ is an anti-directed orientation of~$H$
  and~$m(G)<\mm(H)$, then~$G \narr \Harr$.
\end{corollary}

\begin{proof}
  We have
  $(e(H-u)-1)/(v(H-u)-2) < (e(H)-1)/(v(H)-2)$
  for all~$u \in V(H)$, since~$H$ is strictly $2$-balanced.
  It follows that~$\mm(H)<\delta(H)$,
  so~$\lfloor m_2(H) \rfloor +1 \le \delta(H)$.
  Since~$m(G) < \mm(H) \le \lfloor m_2(H) \rfloor +1/2 \le \delta(H) -1/2$,
  we can apply Theorem~\ref{t:anti-directed}.
\end{proof}

\subsection{Cycles}
\label{s:cycles}

We now consider orientations of $\ell$-cycles, where~$\ell\ge 4$.
The main results are Theorems~\ref{t:ell-ge-5}~and~\ref{t:ell=4},
which deal with the cases $\ell\ge 5$ and~$\ell=4$,
respectively.
(We also include a simple proof for the case $\ell\ge 8$,
see~Theorem~\ref{t:ell-ge-8}.)

\begin{lemma}\label{l:long-block}
  Let $\Carr$ be an oriented cycle with a long block.
  If $G$ is a graph and $m(G)< \mm(\Carr)$, then~$G\narr \Carr$.
\end{lemma}

\begin{proof}
  Note that~$v(\Carr)\ge 4$,
  so~$\mm(\Carr) \le \mm(C_4) = 3/2$. By~Lemma~\ref{l:d-degenerate},
  $G$~is \mbox{$2$-degenerate},
  hence~$\chi(G)\le 3$. Fix a proper colouring
  $c\colon V(G)\to \{1,2,3\}$, and orient each edge towards its
  endvertex with the largest colour. This orientation contains no
  long block, so~$G\narr \Carr$.
\end{proof}

While the next result is superseded by Theorem~\ref{t:ell-ge-5},
its proof is much simpler.

\begin{theorem}\label{t:ell-ge-8}
  Let $\Carr$ be an orientation of~$C_\ell$, where $\ell\ge 8$.
  If $G$ is a graph and~$m(G)< \mm(\Carr)$, then~$G\narr \Carr$.
\end{theorem}

\begin{proof}
  Let $\ell=e(\Carr)$.
  If~\Carr\ contains a long block, then the theorem holds by
  Lemma~\ref{l:long-block}, so we assume that the longest block
  of~\Carr\ has length at~most~$2$.

  Suppose, looking for a contradiction, that the statement is false.
  Without loss of generality let~$G$ be a minimal counterexample
  (with respect to the subgraph relation).
  That is, \mbox{$m(G)<\mm(\Carr)$} and~$G\arr \Carr$,
  and $G'\narr \Carr$ for each proper subgraph~$G'\subseteq G$.
  Let $W$ be the set of vertices in~$G$ with degree~$2$.

  If there exists an edge~$uv$ joining vertices $u,v\in W$, then (since
    $G$ is minimal) $uv$~lies in an $\ell$-cycle.
    Moreover,~$G \setminus\{u,v\}\narr \Carr$; so there exists an
    orientation~$\vec G$ of~$G\setminus \{u,v\}$ which avoids~\Carr.
    Note that each $\ell$-cycle in~$G$ is either completely oriented in~$\vec G$
    (while avoiding \Carr), or contains the three (not yet oriented) edges incident
    with either~$u$ or~$v$.
    We extend~$ \vec G$ by orienting these edges so that they form
    a directed path or cycle. Since, by assumption, the length of any block
    of~$\Carr$ is at most two, it follows that~$\vec G$ is an
    orientation of~$G$ avoiding~\Carr, a contradiction.

  Hence no edge of~$G$ lies in~$W$.
  By the minimality of~$G$,
  every vertex $v\in V(G)$ lies an $\ell$-cycle, so
  $\delta(G)\ge 2$. Let~$n \deq v(G)$.
  Since each vertex of~$W$ has degree~$2$,
    \begin{equation*}
        2\abs{W}  + 3(n - \abs{W})
        \le     \sum_{v \in V(G)} d(v)
        =        2e(G)
        \le     2m(G)n
        <       2n\frac{\ell -1}{\ell-2}.
    \end{equation*}
    It follows that~$\abs{W} \ge n(1-2/(\ell-2))$ and
    \begin{equation*}
        2\left(1-\frac{2}{\ell-2}\right)n
        \le    2\abs{W}
        \le    e(G)
        \le    m(G)n
        =       \left(1+\frac{1}{\ell-2}\right)n,
    \end{equation*}
    which is a contradiction for $\ell \ge 8$.
\end{proof}

\subsubsection{Cycles of length at least~5}
\label{s:ell-ge-5}

We now generalise Theorem~\ref{t:ell-ge-8} for
oriented cycles with at least 5~vertices.

\begin{theorem}\label{t:ell-ge-5}
  Let \Carr\ be an
  orientation of~$C_\ell$, where~$\ell \ge 5$.
  If $G$ is a graph and~$m(G) < \mm(C_\ell)$, then~$G\narr \Carr$.
\end{theorem}

Barros, Cavalar, Mota and Parczyk~\cite{BPCOMP20:antiramsey} obtained
a detailed characterisation of the construction sequences of
$C_\ell$-components; we state their result below in a slightly
modified form (the original has `$\ell\ge 5$' in place
of~`$\ell\ge 4$', but the same proof holds).

\begin{proposition}[{\cite[Proposition 7]{BPCOMP20:antiramsey}}]\label{prop_config}
	Let $\ell\ge 4$ be an integer, $G$ be a graph with $m(G)<\mm(C_\ell)$ and $(H_1,\ldots, H_t)$ be a $C_\ell$-component of $G$.
	The following holds for every $1 \le i \le t-1$.
	If~$C$ is an~$\ell$-cycle added to~$H_{i}$ to form~$H_{i+1}$,
	then there exists a labelling $C=u_1u_2\cdots u_\ell u_1$ such that
	exactly one of the following occurs, where $2 \le j \le \ell$ and $3 \le k \le \ell-1$.

	\begin{enumerate}
        \item[$(A_j)$]\label{config_A}
          $u_1u_2\cdots u_j$ is a~$j$-path in~$H_{i}$ and $u_{j+1},\dots,u_\ell \notin V(H_i)$;

        \item[$(B_k)$] \label{config_B}
          $u_1u_2 \in E(H_i)$, $u_2u_3 \notin E(H_i)$, $\{u_3,\dots, u_\ell\} \setminus \{u_k\} \subseteq V(H_{i+1})\setminus V(H_i)$, $u_k \in V(H_i)$.
	\end{enumerate}
\end{proposition}

If $(H_1,\ldots,H_t)$ constructs a
$C_\ell$-component, then
for each $i\in[t-1]$ the new edges
in~$H_{i+1}$ form a path (by~Proposition~\ref{prop_config}).
We denote this path by~\defi{$Q_i$}, write
\defi{$x_i,z_i$} for its endvertices and \defi{$y_i$}~for the sole internal
vertex of~$Q_i$ in~$H_i$, if it exists. (Again by~Proposition~\ref{prop_config},
$V(Q_i)\cap V(H_i)$ is either
$\{x _i,z_i\}$ or~$\{x_i,y_i,z_i\}$.) We write \defi{$\type(i)$} to
denote the operation (\A{j}~or~\B{k}, where~$2 \le j \le \ell$
and~$3\le k \le \ell -1$) which constructs~$H_{i+1}$ from~$H_i$.

\begin{proposition}[\cite{BPCOMP20:antiramsey}]\label{p:config_bounds}
  Let~$\ell\ge 5$. If $G=(H_1,\ldots,H_t)$ is
  a~$C_\ell$-component and $m(G)<\mm(C_\ell)$, then for all distinct $i,j\in [t-1]$
  and each~$k\in\{3,\ldots,\ell -1\}$ we have the following.
  \begin{itemize}
  \item If $\type(i)= \A{\ell}$, then every other step is of~type
    \A{2}~or~\A{3}.

  \item If $\type(i)=\A{\ell -1}$, then every other
    step is of~type \A{2},~\A{3} or~\Aellminus.

  \item If $\type(i)=\type(j)=\A{\ell -1}$, then $\ell=5$ and every
    other step is of~type~\A{2}.

  \item If $\type(i)=\B{k}$, then every other step is of~type~\A{2}.
  \end{itemize}

\end{proposition}

We also use the following results.

\begin{remark}\label{r:cycle-module}
  Let $G$ be a $C_5$-component. If $G$ can be constructed solely by
  steps of type $(A_2)$, then every cycle in~$G$ has length congruent
  to~$2\pmod 3$.
\end{remark}

\begin{proof}
  The proof is by induction on~$i\in[t]$ where
  $(H_1,\ldots,H_t)$ is the construction sequence of~$G$.
  The base holds because $H_1$~is a~$5$-cycle.
  Now suppose every cycle in $H_i$ has length congruent
  to $2\pmod 3$, where $i\ge 1$. We form~$H_{i+1}$ by a step of
  type~$(A_2)$, i.e., by adding an $5$-path~$P$ joining the endvertices
  of an edge~$uv$ of~$H_i$. Any new cycle~$C$ is formed by an
  $uv$-path~$P'$ in~$H_i$, together with~$P$. If~$P'=uv$,
  then~$C$ has length $5$, and the claim holds. On the other
  hand, if~$uv\notin E(P')$, then~$C=P'\cup P$, but
  since~$C'\deq P'+uv$ is a cycle in~$H_i$, it follows
  that~$e(C')\equiv 2\pmod 3$, so $e(C)=e(P') + e(P) = e(C') -1 + e(P) \equiv
    2\pmod 3$. \qedhere
  \end{proof}

\begin{remark}\label{r:before-cherries}
  Let~$G$ be a $C_5$-component. If~$G$ is
  constructed solely by steps of the types \A{2}~and~\A{3}, then
  $G$~contains no $C_3$~and no~$C_4$.
\end{remark}

\begin{proof}
  Let~$(H_1,\ldots,H_t)$ be a construction sequence
  of~$G$. Note that $C_3\nsubseteq G$: indeed,
  $H_1\simeq C_5$, so~$C_3\nsubseteq H_1$;
  moreover, for each $i\in[t-1]$ we have~$H_{i+1}=H_i\cup Q_i$ and $Q_i$ is a~path
  of length at~least~$3$ which is internally
  disjoint from~$H_i$, so~$C_3 \nsubseteq H_{i+1}$.

  Similarly~$C_4\nsubseteq H_1$, and
  if~$H_i\cup Q_i$ contains a $C_4$,
  then $\type(i)=\A{3}$, so~$x_i$~and~$z_i$ are connected by
  a path $x_iwz_i$ in~$H_i$ (which, together with~$Q_i$,
  creates a~$C_5$). But then $x_iwz_ix_i$ is a~$C_3$
  in~$G$, a contradiction.
\end{proof}

We are now in position to prove the main result of this section.

\begin{proof}[{Proof of Theorem~\ref{t:ell-ge-5}}]
  Let~$G$ be a graph with~$m(G) < (\ell-1)/(\ell-2)$, where
  $\ell \ge 5$, and let $\Carr$ be an oriented $\ell$-cycle.
  By~Lemma~\ref{l:long-block}, if \Carr\ contains a long block,
  then $G\narr \Carr$, so we may
  assume that every block of~\Carr\ has length at most two.
  We will show that the $C_\ell$-components of~$G$
  admit an orientation in which every $\ell$-cycle has a long block.
  It suffices to consider one such component~$F$, as
  $C_\ell$-components can be independently oriented (they
  do not share edges) and remaining edges can be arbitrarily
  oriented (each $\ell$-cycle in~$G$ lies in some $C_\ell$-component).

  Let~$F=(H_1,\dots,H_t)$ be a~$C_\ell$-component of~$G$.
  Hence, for all~$i\in[t-1]$,
  each $\ell$-cycle $C\subseteq H_{i+1}$ which did not exist in~$H_i$
  contains either the path $x_iQ_iy_i$ or~$y_iQ_iz_i$
  (if~$Q_i$ intersects~$H_i$ in three vertices)
  or the whole path~$Q_i$.

  \case{0} For each $i\in[t-1]$ we have
  $\type(i)\notin\{\A{\ell-1},\A{\ell},\B{3},\ldots,\B{\ell
    -1}\}$.

  For each~$i\in[t-1]$, every new cycle in~$H_{i+1}$ contains~$Q_i$
  and~$\length(Q_i)\ge 3$. We construct an orientation of~$F$ which
  avoids~\Carr\ as follows. Fix a~directed orientation of~$H_1$, and
  for each~$i\in[t-1]$ fix a~directed orientation of~$Q_i$.
  Clearly~$H_1$ does not contain~\Carr, and for each~$i\in[t-1]$ every new $\ell$-cycle
  in~$H_{i+1}$ contains a long block (since~$Q_i$ is directed),
  so~$F\narr\Carr$.

  \case{1} There is precisely one index $i\in[t-1]$ such that~$\type(i) = \Aellminus$.

  Let~$Q_i=x_ivz_i$ and let~$C$ be an
  $\ell$-cycle in~$H_i$ containing~$z_i$. We may
  assume that~$H_1= C$. Note that ~$\length(Q_j)\ge 3$
  for each $j\in[t-1]\setminus\{i\}$ since
  (by~Proposition~\ref{p:config_bounds}) $\type(j)\in\{\A{2},\A{3}\}$.
  We orient~$F$ as follows.

  Firstly, orient~$H_1$ so that $z_i$~is the origin of a long block,
  and so that $z_i$~has no~inneighbours in~$H_1$.
  Secondly, for each $j\in[i-1]$,
  orient~$Q_j$ forming a directed path, while ensuring that $z_i$~has no
  inneighbours in~$H_{j+1}$. (This is possible since,
  if $Q_j$ contains~$z_i$, then $z_i$~is an~endvertex of~$Q_j$.)
  Orient $Q_i$ as a directed path from~$x_i$ to~$z_i$.
  Finally, for each~$j\in[t-1]\setminus[i]$ orient~$Q_j$ so as to form
  a~directed path.

  Clearly, the orientation of~$H_1$ avoids~\Carr. Since
  $\length(Q_j)\geq 3$ for each $j\in[t-1]\setminus\{i\}$,
  each new $\ell$-cycle in~$H_{j+1}$ has a long block (as it contains~$Q_j$).
  Finally, every new cycle~$C$ in~$H_{i+1}$ must contain $Q_i$ as well as
  some edge $z_iz\in E(H_i)$. As $z_i$ has no inneighbours in~$H_i$, the
  edge $z_iz$ extends the directed path $x_i\arr v\arr z_i$,
  forming a long block in~$C$.
  This shows that every $\ell$-cycle has a long block, so~$F\narr\Carr$.

  \case{2} There exists $i\in[t-1]$ such that~$\type(i)=\A{\ell}$.

  Let~$\alpha \in[t-1]$. By Proposition~\ref{p:config_bounds}, if
  $\alpha\neq i$, then~$\type(\alpha)\in\{\A{2},\A{3}\}$, so
  $\length(Q_\alpha)\ge 3$. We may assume that~$H_1$ is an
  $\ell$-cycle in~$H_i$ containing $z_i$.
  We orient the edges of~$F$ as follows.
  Let~$N$ be the set of neighbours of~$z_i$ in~$H_i$.

  \afigure{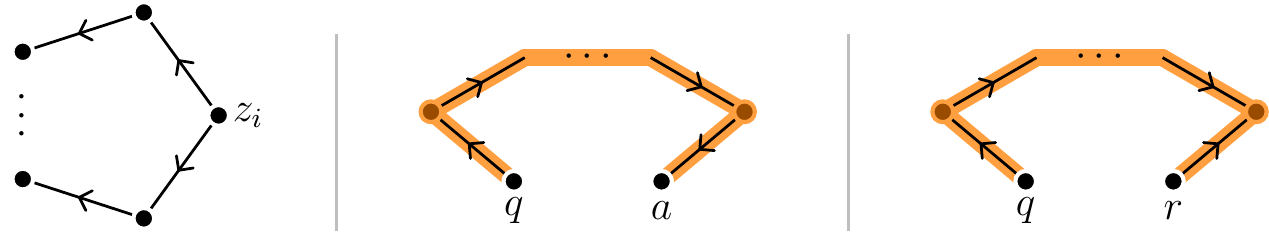}{Orientations in~Case~2. Left: orientation of~$H_1$;
    note~$H_1$ has a long block starting from~$z_i$ (since $\ell\geq 5$).
    Centre and right: orientations of~$Q_\alpha$ (where~$\alpha\neq i$);
    in the figure, $a\notin \{z_i\}\cup N$, $r\in N$ and~$q\in \{z_i\}\cup N$,
    where~$N\deq N_{H_i}(z_i)$.}{fig:case1}

  First orient~$H_1$ with two blocks, each with length at least~$2$
  and origin~$z_i$ (see Figure~\ref{fig:case1}).
  Next, for each $j\in[i-1]$, we do the following.
  \textbf{If no endvertex of~$Q_j$ lies in~$\{z_i\}\cup N$},
    fix an arbitrary directed orientation of~$Q_j$.
    \textbf{If a single endvertex~$q$ of~$Q_j$ lies in~$\{z_i\}\cup N$}, then
    orient~$Q_j$ to form a directed path with origin~$q$.
  \textbf{If both endvertices $q,r$ of~$Q_j$ lie in~$\{z_i\}\cup N$}, where we
  assume~$r\neq z_i$, then orient~$Q_j$ so that it has precisely two blocks,
  starting from $q$~and~$r$,
  and so that the latter has precisely one~arc.
  Finally, orient~$x_i\arr z_i$, and for
  each $j\in[t-1]\setminus[i]$ fix a directed orientation of~$Q_j$
  (see Figure~\ref{fig:case1}).

  Let us check that every $\ell$-cycle in~$F$ has a long block.
  This is clearly true in~$H_1$.
  Now suppose~$\alpha\in[t-1]\setminus \{i\}$.
  Note that each new cycle in~$H_{\alpha+1}$
  contains~$Q_\alpha$ and that~$\length(Q_\alpha)\geq 3$ since
  $\type(\alpha)\in\{\A{2},\A{3}\}$.
  Moreover, $Q_\alpha$ has a block of length at least~$\length(Q_\alpha)-1$
  if $\alpha < i$,
  and a block of length at~least~$\length(Q_\alpha)$ if~$\alpha > i$.
  Hence, if $\length(Q_\alpha)\ge 4$ or if~$\alpha > i$,
  then $Q_\alpha$~has a long block.
  So we may suppose that $\ell=5$, $\length(Q_\alpha) = 3$
  and~$\alpha\in[i-1]$.
  Hence~$\type(\alpha)=\A{3}$ and there is precisely one new
  $5$-cycle~$C$ in~$H_{\alpha +1}$ (as otherwise
  two $3$-paths joining $x_\alpha$~and~$z_\alpha$, would form a
  $4$-cycle in~$H_\alpha$, contradicting
  Remark~\ref{r:before-cherries}).
  If~$\bigl|\{x_\alpha,z_\alpha\}\cap (\{z_i\}\cup N)\bigr|\le 1$, then
  $C$~has a long block containing~$Q_\alpha$.
  Otherwise,~$\{x_\alpha,z_\alpha\}\subseteq \{z_i\}\cup N$.
  Note that $x_\alpha z
  z_\alpha\subseteq H_\alpha$ for some~$z\in V(H_\alpha)$ since~$C\subseteq H_{\alpha +1}$;
  if~$z_i\in\{x_\alpha,z_\alpha\}$, then $x_\alpha z_\alpha \in E(H_i)$,
  so $x_\alpha z_\alpha z x_\alpha$ is a triangle in~$H_i$,
  contradicting Remark~\ref{r:before-cherries}.
  Therefore $z_i\notin\{x_\alpha,z_\alpha\}$,
  so~$C= Q_\alpha\cup x_\alpha z_iz_\alpha$
  (since~$z\neq z_i$ implies $x_\alpha z z_\alpha z_i x_\alpha$
  is a $4$-cycle in~$H_i$,
  which contradicts Remark~\ref{r:before-cherries}).
  Since $Q_\alpha$ has a directed $3$-path from either
  $x_\alpha$~or~$z_\alpha$ to a
  vertex~$w\in V(Q_\alpha)\setminus V(H_\alpha)$, and both
  $x_\alpha$~and~$z_\alpha$ are outneighbours of~$z_i$,
  it follows that~$C$ has a long block.

  To conclude Case~2, we consider the new $\ell$-cycles
  in~$H_{i+1}$. Each of these cycles contains the arc~$x_i\to z_i$, so it suffices to show that
  every $3$-path $z_izw$ in $H_{i}$ is directed from $z_i$~to~$w$.
  Note that for each~$j\in[i-1]$ and
  each pair of distinct new edges $e_1,e_2$ in~$H_{j+1}$,
  there exist distinct new vertices $v_1\in e_1,v_2\in e_2$
  in~$H_{j+1}$.
  It follows that either
  $z_izw\subseteq H_1$; or
  $zw\subseteq Q_\alpha$ and $z$ is an endvertex of~$Q_\alpha$ for some~$\alpha\in[i-1]$; or
  $z_izw\subseteq Q_\beta$ and $z_i$ is an endvertex of~$Q_\beta$ for some~$\beta\in[i-1]$.
  In each of these cases $z_izw$ has the required orientation.

  Since every $\ell$-cycle of $F$ is a long block and~$F\narr\Carr$.

  \case{3} There exist $i,j\in[t-1]$ such
  that~$\type(i)=\type(j)=\Aellminus$.

  \afigure{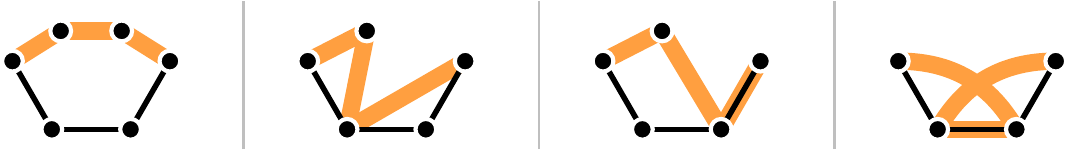}{Unions of distinct $4$-paths
    with common endvertices.}{fig:4-path-union}

  By~Proposition~\ref{p:config_bounds} we have~$\ell=5$ and
  $\type(\alpha)=\A{2}$ for each~$\alpha\in[t-1]\setminus\{i,j\}$.
  We may suppose~$i < j$.
  Let $P=x_iu_2u_3z_i \subseteq H_i$ and~$Q=x_jv_2v_3z_j\subseteq H_j$, and let
  $Q_i=x_iu_5z_i$~and $Q_j=x_jv_5z_j$. By Remark~\ref{r:cycle-module}, every
  cycle in~$H_i$ has length congruent to~$2$~modulo~$3$, so $H_i$
  contains no~$C_3$, no~$C_4$ and no~$C_6$. In particular, since the
  union of internally disjoint $4$-paths with common ends
  contains~$C_3,C_4$ or~$C_6$
  cycle of length~$3,4$~or~$6$ (see Figure~\ref{fig:4-path-union}), we
  conclude that $P$ is the unique $4$-path between $x_i$~and~$z_i$
  in~$H_i$, and hence the unique such path in~$H_{i+1}$.
  The argument splits into three cases according to how the $5$-cycles
  in~$H_i$ intersect~$P$.

  \subcase{a} There exists a $5$-cycle~$C$ in~$H_i$ containing~$P$.

  We
  may assume~$H_1= C$ and~$i=1$. Let~$C=x_iu_2u_3z_ixx_i$
  (so $H_2=H_{i+1}=C\cup x_iu_5z_i$).
  We first prove that
    \begin{equation}\label{e:case-3-forb-cycles}
    \text{$H_j$ contains  no~$C_3$, no~$C_6$, and precisely one $C_4$}.
  \end{equation}
    Crucially, note that a
    step of type~\A{2} cannot create a~$C_3$ or a~$C_4$. Therefore,
    since~$H_1\simeq C_5$, each~$C_3$ and each~$C_4$ in~$H_j$ were
    created in the $i$-th step resulting in~$H_{i+1}$.
    Since~$H_{i+1}$ is the union of~$C$ and~$z_iu_5x_i$, we conclude
    that~$C_3\nsubseteq H_{i+1}$, so~$C_3\nsubseteq H_j$; moreover,
    the unique~$C_4\subseteq H_j$ is~$x_ixz_iu_5x_i$.
    It remains to show that~$H_j$ contains no~$C_6$. Suppose, looking for
    a contradiction, that
    $\alpha\in [j-1]$ is the smallest index such that
    $H_{\alpha+1}$ has a $6$-cycle~$C'$.
    Note that~$H_{i+1}$
    contains no~$C_6$, so~$\alpha>i$. Since $\type(\alpha)=\A{2}$,
    it follows that $C'$ contains a path~$abcde$ whose edges
    are new in~$H_{\alpha+1}$, so $C'=\text{\textit{abcdefa}}$ for
    some~$f\in V(H_\alpha)$. Moreover, $abcdea$ is a~(new) $5$-cycle
    in~$H_{\alpha+1}$. We conclude that $\text{\textit{aefa}}$ is a
    $3$-cycle in~$H_\alpha$, a contradiction since $C_3\nsubseteq H_j$.
    This proves~\eqref{e:case-3-forb-cycles}.

  \begin{claim}\label{cl:common-edge}
    There exists~$e\in E(H_j)$ with~$e\cap \{x_j,z_j\}\neq \emptyset$ which lies in every $4$-path from
    $x_j$~to~$z_j$ in~$H_j$.
  \end{claim}

  \begin{claimproof}
    By~\eqref{e:case-3-forb-cycles}, each $4$-path between $x_j$~and~$z_j$ in~$H_j$
    other than~$Q$ intersects~$x_jv_2v_3z_j$~(i.e.,~$Q$) in precisely one edge~$h$;
    moreover,~$h\neq v_2v_3$
    (as~$C_3\nsubseteq H_j$, see Figure~\ref{fig:4-path-union}).
    If Claim~\ref{cl:common-edge} is false,
    then there are paths $x_jxv_3z_j$ and~$x_jv_2yz_j$ in~$H_j$
    with $x\neq v_2$~and~$y\neq v_3$.
    But this contradicts~\eqref{e:case-3-forb-cycles},
    because then either
    $H_j$ has a $3$-cycle $x_jxv_2x_j$ (if~$x=y$)
    or~$H_j$ contains two distinct $4$-cycles
    $x_jxv_3v_2x_j$~and~$v_2v_3z_jyv_2$ (if~$x\neq y$).
\end{claimproof}

  We now return to the proof of \emph{Case (a)},
  describing the orientation of~$F$.
  Let $e$ be the edge common to all~$4$-paths between~$x_j$ and~$z_j$
  in~$H_j$ (as per~Claim~\ref{cl:common-edge}).
  Orient~$H_1$ so that it is a directed cycle.
  For every $\alpha\in[t-1]\setminus\{j\}$, orient the new edges to
  form a directed path. Finally,
  orient~$x_jv_5z_j$ so that the path it forms with~$e$ is directed.

  Let us check that every $5$-cycle in~$F$ has a long block.
  Clearly, the two $5$-cycles in $H_2$~have each a long block.
  For each $\alpha\in[t-1]\setminus\{i,j\}$,
  each new $5$-cycle in~$H_{\alpha+1}$
  contains~$Q_\alpha$ and hence has a long block
  ($e(Q_\alpha)\ge 3$ since~$\type(\alpha)=\A{2}$).
  Finally, every new $5$-cycle
  in~$H_{j+1}$ contains the directed path formed by
  $e$~and~$x_jv_5z_j$. We conclude that~$F\narr\Carr$.

  \subcase{b} There exists a $5$-cycle~$C$ in~$H_i$ containing precisely two
  edges of~$P$.

  We may assume that no $5$-cycle in~$H_i$
  contains all edges of~$P$, otherwise we would be done by \emph{Case
    (a)}.
  Note that~$C$ cannot avoid~$u_2u_3$, since~$C_3\nsubseteq H_i$.
  We may therefore assume that $C$
  is a $5$-cycle in~$H_i$ with~$z_iu_3u_2\subseteq C$ and that~$H_1=C$.

  Let $\alpha\in[t-1]$ be such that $u_2x_i$ is new in~$H_{\alpha+1}$,
  and let~$C_\alpha$ be a new $5$-cycle
  in~$H_{\alpha+1}$ containing~$u_2x_i$.
  Note that $\type(\alpha)=\A{2}$,
  so $Q_\alpha=u_2x_ixyv$, where
  $u_2,v\in V(H_{\alpha})$~and $x_i,x,y\notin V(H_\alpha)$. We modify the
  construction sequence of~$F$, to a construction
  sequence of~$F$ where
  the $i$-th step is omitted and the $\alpha$-th step
  is replaced by consecutive steps adding, in this order,
   $u_2x_iu_5z_i$~and~$x_ixyv$.
  In the new sequence, $\type(\alpha)=\type(\alpha+1)=\A{3}$,
  $\type(j)=\A{4}=\Aellminus$ and each other step remains of type~\A{2}.
  By the argument in Case~2,
  $F\narr\Carr$.

  \subcase{c} Every $5$-cycle in~$H_i$ contains at most one edge
  of~$P$.

  This is similar to the preceding case.
  Let~$C=H_1$ be a $5$-cycle containing~$z_iu_3$.
  We first show that
  if~$u_2u_3$ is new in~$H_{\alpha+1}$ and~$u_2x_i$ is
  new~$H_{\beta+1}$, then~$\alpha < \beta < i$.
  Indeed, $\alpha,\beta<i$ by definition, and
  $\alpha\neq \beta$ as otherwise the new cycles in
  $H_{\alpha+1}$ would contain two edges of~$P$.
  Moreover, $\type(\alpha)=\type(\beta)=\A{2}$
  by Proposition~\ref{p:config_bounds},
  so each new
  edge in~$H_{\alpha+1}$ and~$H_{\beta+1}$ must contain at least one new endvertex.
  Hence~$\alpha< \beta$.

  Let $Q_\beta=u_2x_ixyv$, where $u_2,v\in V(H_\beta)$
  and~$x_i,x,y\notin V(H_\beta)$. As in \emph{Case~(b)},
  we define an alternative construction sequence of~$F$,
  where the $i$-th step is omitted and the $\beta$-th step
  is replaced by consecutive steps adding
  $u_2x_iu_5z_i$~and~$x_ixyv$ (in this order).
  By~Case~2,~$F\narr\Carr$.

  \case{4} There exists~$i\in[t-1]$ such that~$\type(i)=\B{j}$,
  where $3\le j\le \ell -1$.

  By~Proposition~\ref{p:config_bounds},
  for each~$\alpha\in[t-1]\setminus \{i\}$ we have~$\type(\alpha)=\A{2}$,
  and thus $e(Q_\alpha)\ge 3$.
  Recall that $y_i\in V(Q_i)\cap H_i$.
  Note that no new cycle in~$H_{i+1}$ avoids
  both~$x_iQ_iy_i$ and~$y_iQ_iz_i$.

  If a new $\ell$-cycle in~$H_{i+1}$ contains $x_iQ_iy_i$ but
  not~$y_iQ_iz_i$, then some construction sequence of~$F$ satisfies
  the hypothesis of one of the previous cases
  (by~replacing the $i$-th step in $(H_1,\ldots,H_t)$
  by consecutive steps adding $x_iQ_iy_i$ and~$y_iQ_iz_i$), and~$F\narr\Carr$.
  We argue similarly if a new $\ell$-cycle in~$H_{i+1}$ avoids $x_iQ_iy_i$.

  If every new $\ell$-cycle in~$H_{i+1}$ contains all of~$Q_i$, then for each
  $\alpha\in[t-1]$ every new cycle in~$H_{\alpha+1}$
  contains~$Q_\alpha$. We fix a directed orientation of~$H_1$
  and orient~$Q_\alpha$ as a directed path for each $\alpha\in[t-1]$.
  Then $H_1$ has a~long block and for each~$\alpha\in[t-1]$
  the new $\ell$-cycles in~$H_{\alpha+1}$ have
  a long block as well (since~$\length(Q_\alpha)\ge 3$). Therefore~$F\narr\Carr$.
\end{proof}

\subsubsection{Cycles of length 4}
\label{s:ell=4}

To conclude this section we consider orientations of $4$-cycles.

\begin{theorem}\label{t:ell=4}
  Let \Carr\ be an orientation of~$C_4$. If~$G$ is a graph and~$m(G)<\mm(C_4)$, then $G\narr \Carr$.
\end{theorem}

To prove Theorem~\ref{t:ell=4} we use the following proposition.

\begin{proposition}\label{p:config_bounds-C4}
  Let $G=(H_1,\ldots,H_t)$ be a $C_4$-component such that $m(G)< \mm(C_4)$. If $\type(i)= \B{3}$ for some $i$, then~$\type(j)=\A{2}$ for each $j\in [t-1]\setminus\{i\}$.
\end{proposition}

\begin{proof}
  For each $j\in[t-1]$, let $v_j$ and~$e_j$ be respectively the number of new vertices and new~edges in~$H_{j+1}$,
  By Proposition~\ref{prop_config} we have~$e_j\ge 3v_j/2$ and~$e_j>v_j$ for each~$j\in[t-1]$.
  Suppose~$\type(i)=\B{3}$ and fix~$j\in[t-1]\setminus\{i\}$. We have
  \[
    \frac{3}{2}
    =    \mm(C_4) > m(G)
    =    \frac{4 + \sum_{\alpha\in[t-1]} e_\alpha}{4 + \sum_{\alpha\in[t-1]} v_\alpha}
    \ge  \frac{4 + 3 + e_j + \sum_{\alpha\in[t-1]\setminus\{i,j\}} 3v_\alpha/2}{4 + 1 + v_j + \sum_{\alpha\in[t-1]\setminus\{i,j\}} v_\alpha},
  \]
  so~$v_j > 2(e_j-v_j) -1$. Hence~$v_j\ge 2$ (because~$v_j<e_j$) and~$\type(j)=\A{2}$.
\end{proof}

\begin{proof}[Proof of Theorem~\ref{t:ell=4}]
  If $\Carr$ is anti-directed or contains a long block, then $G\narr \Carr$ by~Corollary
  \ref{cor:anti-directed}~and Theorem~\ref{l:long-block}, respectively. We may
  therefore assume~\Carr\ has precisely two blocks of length~$2$; we
  may also assume that~$G$ is a $C_4$-component with construction
  sequence~$(H_1,\ldots,H_t)$,
  because distinct $C_4$-components can be independently oriented
  and edges in no $C_4$-component can be arbitrarily oriented.

  If there is no step of type~$B_3$, then~$G$ is bipartite.
  (Indeed, $H_1\simeq C_4$ and
  steps of type~\A{2},\A{3}, or~\A{4} preserve bipartiteness.)
  Fix a proper $2$-colouring of~$G$ and orient every edge towards the same colour class.
  This avoids directed paths with length~$2$, so $G\narr\Carr$.

  On the other hand, if $\type(i)=\B{3}$, then
  every other step is of type~\A{2} by Proposition~\ref{p:config_bounds-C4}.
  Let~$u_1u_2u_3u_4u_1$ be the new cycle in~$H_{i+1}$, where
  $u_1u_2\in H_i$ (and~$u_2u_3,u_3u_4,u_4u_1\notin E(H_i)$, $u_1,u_2,u_3\in V(H_i)$,
  $u_4\notin V(H_i)$). We may
  assume that $H_1$~is a $4$-cycle $u_1u_2abu_1$.

  If every new $4$-cycle in~$H_{i+1}$
  contains~$u_2u_3u_4u_1$, we orient $H_1$ as a directed
  cycle and the new edges in each step as~directed paths.
  Clearly $H_1$ has a long block and, for each~$\alpha\in[t-1]$,
  every new $4$-cycle in~$H_{\alpha +1}$ contains a long block (formed by~$Q_\alpha$),
  so~$G\narr\Carr$.

  Finally, if a $4$-cycle in~$H_i$ contains $u_2u_3$ but avoids
  $u_3u_4u_1$, then we may replace the $i$-th step (of type~\B{3}) by one
  \A{4}-step (adding~$u_2u_3$) and one~\A{3}-step
  (adding~$u_3u_4u_1$). This yields a construction sequence free
  from~\B{3}, which implies (as argued above) that~$G$ is bipartite
  and~$G\narr \Carr$. Similarly, if $H_i$ contains a new
  $4$-cycle which avoids $u_2u_3$, then we may replace the $i$-th step
  by one \A{3}-step (adding $u_3u_4u_1$) and one \A{4}-step
  (adding~$u_2u_3$), and also conclude that~$G\narr\Carr$.
\end{proof}

\section{Proof of the main theorem (Theorem~\ref{t:thresholds})}
\label{s:main-proofs}

Theorem~\ref{thm:bound_ttt} establishes the case~$t=3$
of Theorem~\ref{t:thresholds}.
We may therefore suppose \Harr\ is either
an acyclic orientation of~$H\in\{K_t,C_t\}$,
with~$t \ge 4$,
or that \Harr\ is an anti-directed orientation of
a strictly $2$-balanced graph~$H$ with~$\delta(H)\ge 2$.
In each one of these cases
\Harr\ is $2$-Ramsey-avoidable (by Remark~\ref{r:2-ramsey-avoidable}),
so (by Theorems \ref{thm:cav18_threshold}~and~\ref{t:npss17_block} together with
Lemmas \ref{l:decomp-in-blocks}~and~\ref{l:component-to-blocks})
it suffices to show
that $G\narr\Harr$ whenever $m(G)<\mm(H)$.
Indeed, this follows by Theorem~\ref{t:tourn-dens-lbound} (when~$H$ is complete),
Theorems~\ref{t:ell-ge-5} and~\ref{t:ell=4} (when $H$ is a~cycle) and by
Corollary~\ref{cor:anti-directed} otherwise.\hfill\qedsymbol

\section{Concluding remarks}
\label{s:concl}

We have shown that if~$\Harr$ is an oriented clique or cycle, then the
threshold for~$G(n,p)\arr\Harr$ is~$n^{-1/\mm(\Harr)}$ if and only
if~$\Harr\neq \ttt$. Interestingly, \ttt\ is not the only exception.
For instance, let $\Garr$ be the digraph obtained from an oriented
tree~$\Tarr$ of order~$n^{1/2-\eps}$, for any fixed $\eps>0$, by
identifying with each $v\in V(\Tarr)$ the source of a distinct copy~$\Harr_v$
of~\ttt. It can be shown that~$p_{\Garr} \ll n^{-1/\mm(\Garr)} = n^{-1/\mm(\ttt)}$. In a
forthcoming paper, the authors describe a richer class of digraphs
with this property.

\bibliographystyle{amsplain}

{\footnotesize

  \bibliography{orient-ramsey-thresh.bib}\par

  \bigskip

  \bigskip

  \noindent
  \textsc{Gabriel Ferreira Barros, Yoshiharu Kohayakawa, T\'assio Naia}\newline
  \textsc{Instituto de Matemática e Estatística, Universidade de São Paulo, São Paulo, Brazil}\newline
  {\scriptsize\texttt{\{gbarros, yoshi, tassio\}@ime.usp.br}}

  \bigskip

  \noindent
  \textsc{Bruno Pasqualotto Cavalar}\newline
  \textsc{University of Warwick, United Kingdom}\newline
  {\scriptsize\texttt{bruno.pasqualotto-cavalar@warwick.ac.uk}}

\vspace*{\fill}
\hfill December 2020

}
\end{document}